\documentclass[11pt,reqno]{amsart}
\usepackage[vmargin=2cm,hmargin=2cm]{geometry}
\usepackage{amssymb, latexsym, amsfonts,amsmath,amsthm,color,etaremune,enumerate,tikz}
\usepackage{caption,float,textcomp,units,xfrac}
\allowdisplaybreaks

\title{Asymptotics and inequalities for partitions into squares}
\author{Alexandru Ciolan}
\address{Mathematical Institute, University of Cologne, Weyertal 86--90, 50931
	Cologne, Germany}
\email{aciolan@math.uni-koeln.de}
\newtheorem{Thm}{Theorem}
\newtheorem{Con}{Conjecture}

\newtheorem{Lem}{Lemma}

\newcommand{\bb}{\mathbb}

\theoremstyle{remark}

\newtheorem{Rem}{Remark}

\DeclareMathOperator{\Arg}{Arg}
\DeclareMathOperator{\Res}{Res}
\DeclareMathOperator{\Log}{Log}
\makeatletter
\let\@@pmod\pmod
\DeclareRobustCommand{\pmod}{\@ifstar\@pmods\@@pmod}
\def\@pmods#1{\mkern4mu({\operator@font mod}\mkern 6mu#1)}
\makeatother

\allowdisplaybreaks

\begin{document}
\begin{abstract}
In this paper we prove that the number of partitions into squares with an even number of parts is asymptotically equal to that of partitions into squares with an odd number of parts. We further show that, for $ n $ large enough, the two quantities are different and which of the two is bigger depends on the parity of $ n. $ This solves a recent conjecture formulated by Bringmann and Mahlburg (2012). 
\end{abstract}
\maketitle

\section{Introduction}\label{Intro}
A \textit{partition} of a positive integer $ n $ is a non-increasing sequence of positive integers (called its \textit{parts}), usually written as a sum, which add up to $ n. $ The number of partitions of $ n $ is denoted by $ p(n). $ For example, $ p(5)=7 $ as the partitions of $ 5 $ are $ 5, $ $ 4+1, $ $ 3+2, $ $ 3+1+1, $ $ 2+2+1, $ $ 2+1+1+1 $ and $ 1+1+1+1+1. $ By convention, $p(0)=1.$ This is the case of the so-called \textit{unrestricted} partitions, but one can consider partitions with various other properties, such as partitions into odd parts, partitions into distinct parts, etc. 
\par Studying congruence properties of partition functions fascinated many people and we limit ourselves to mentioning the famous congruences of Ramanujan \cite{Ram}, who proved that if $ n\ge0, $ then 
\begin{align*}
	p(5n+4)&\equiv  {0 \pmod*5,}\\
	p(7n+5)&\equiv  {0 \pmod*7,}\\
	p(11n+6)&\equiv  {0 \pmod*{11}.}
\end{align*} 
\par In this paper we study partitions based on their number of parts being in certain congruence classes. For $ r\in\bb N $ let $ p_r(a,m,n) $ be the number of partitions of $ n $ into $ r $-th powers  with a number of parts that is congruent to $ a$ modulo $ m.$ Glaisher \cite{Glaisher} proved (with different notation) that \[p_1(0,2,n)-p_1(1,2,n)=(-1)^n p_{\rm odd}(n),\]
where $ p_{\rm odd}(n) $ denotes the number of partitions of $ n $ into odd parts without repeated parts. \par It is as such of interest to ask what happens for partitions into $ r $-th powers with $ r\ge2, $ and a natural point to start is by investigating partitions into squares. Based on computer experiments, Bringmann and Mahlburg \cite{BM} observed an interesting pattern and conjectured the following.
\begin{Con}[Bringmann--Mahlburg, 2012]\label{ConjBM}~ 		
\begin{enumerate}[{\rm (i)}]	
\item As $ n\to\infty, $ we have \[p_2(0,2,n)\sim p_2(1,2,n).\] 
\item We have
\[\begin{cases}
p_2(0,2,n)>p_2(1,2,n) & \text{if~$n$~is even,}\\
p_2(0,2,n)<p_2(1,2,n) & \text{if~n~is odd.}
\end{cases}\]\end{enumerate}\end{Con}
We build on the initial work done by Bringmann and Mahlburg \cite{BM} towards solving Conjecture \ref{ConjBM}, the goal of this paper being to prove that the inequalities stated in part (ii) hold true asymptotically. In turn, this will show that part (i) of Conjecture \ref{ConjBM} holds true as well. 
\pagebreak\par More precisely, we prove the following.
\begin{Thm}\label{Conj1} ~		
	\begin{enumerate} 	
		\item[{\textup{(i)}}] As $ n\to\infty, $ we have \[p_2(0,2,n)\sim p_2(1,2,n).\] 
		\item[{\textup{(ii)}}] Furthermore, for $ n $ sufficiently large, we have
		\[\begin{cases}
		p_2(0,2,n)>p_2(1,2,n) & \text{if~$ n $~is even,}\\
		p_2(0,2,n)<p_2(1,2,n) & \text{if~$ n $ is odd}.
		\end{cases}\]\end{enumerate}
\end{Thm}
In other words, we prove that the number of partitions into squares with an even number of parts is asymptotically equal to that of partitions into squares with an odd number of parts. However, for $ n $ large enough, the two quantities are always different, which of the two is bigger depending on the parity of $ n. $ Given that asymptotics for partitions into $ r $-th powers (in particular, for partitions into squares) are known due to Wright \cite{WrightIII}, we can make the asymptotic value in part (i) of Theorem \ref{Conj1} precise. We will come back to this after we give the proof of Theorem \ref{Conj1}.
\par As for the structure of this paper, in Sections \ref{Inequalities for partitions with a fixed number of parts} and \ref{Proof of Conjecture 1} we introduce the notation needed in the sequel and do some preliminary work required for the proof of Theorem \ref{Conj1}, which we give in detail in Section \ref{Proof of Theorem 2}.
\section{Preliminaries}\label{Inequalities for partitions with a fixed number of parts}
\subsection{Notation} Before going into details, we recall some notation and well-known facts that will be used throughout. By $ \Gamma(s)$ and $ \zeta(s) $ we denote the usual Gamma and Riemann zeta functions, while by $$ \zeta(s,q)=\sum_{n=0}^{\infty}\frac1{(q+n)^s}\quad(\text{for~}{\rm Re}(s)>1\text{~and~}{\rm Re}(q)>0) $$ we denote the Hurwitz zeta function. For reasons of space, we will sometimes use $ \exp (z) $ for $ e^z. $ Whenever we take logarithms of complex numbers, we use the principal branch and denote it by $ \Log $. By $ \zeta_n=e^{\frac{2\pi i}{n}} $ we denote the standard primitive $ n $-th root of unity.
\par If by $ p_r(n) $ we denote the number of partitions of $ n $ into $ r $-th powers, then it is well-known (see, for example, Andrews \cite[Ch. 1]{And}) that 
\[\prod_{n=1}^{\infty}\left(1-q^{ n^r} \right)^{-1}=1+\sum_{n=1}^{\infty}p_r(n)q^n,\] where, as usual, $ q=e^{2\pi i\tau} $ and $ \tau\in\bb H $ (the upper half-plane).
\subsection{A key identity} Let
\[H_r(w;q)=\sum_{m,n\ge0}p_r(m,n)w^mq^n,\]
where $ p_r(m,n) $ denotes the number of partitions of $ n $ into $ r $-th powers with exactly $ m $ parts, and let
\[H_{r,a,m}(q)=\sum_{n\ge0}p_r(a,m,n)q^n,\]
where $ p_r(a,m,n) $ stands, as defined in the Introduction, for the number of partitions of $ n $ into $ r $-th powers with a number of parts that is congruent to $ a$ modulo $ m. $  

By using the orthogonality of roots of unity, we obtain 
\begin{equation}\label{orthogonal}
	H_{r,a,m}(q)=\frac 1m H_r(q)+\frac 1m \sum_{j=1}^{m-1}\zeta_m^{-aj}H_r(\zeta_m^j;q),
\end{equation}
where we denote  \[H_r(q)=\prod_{n=1}^{\infty}\left(1-q^{ n^r} \right)^{-1}.\]

\subsection{A reformulation of our result}\label{reformulation} 
For the rest of the paper we only deal with the case $ r=2, $ which corresponds to partitions into squares. 
To prove part (ii) of Theorem \ref{Conj1} it is enough to show that the series
\[H_{2,0,2}(-q)-H_{2,1,2}(-q)=\sum_{n=0}^{\infty}a_2(n)q^n\]
has positive coefficients for sufficiently large $ n, $ since 
\begin{equation*}\label{a2fromp2}
	a_2(n)=\begin{cases}
		p_2(0,2,n)-p_2(1,2,n) & \text{if~$ n $~is even,}\\
		p_2(1,2,n)-p_2(0,2,n) & \text{if~$n$~is odd.}
	\end{cases}
\end{equation*}
Using, in turn, \eqref{orthogonal} and eq. (2.1.1) from Andrews \cite[p. 16]{And}, we obtain  
\[H_{2,0,2}(q)-H_{2,1,2}(q)=H_{2}(-1;q)=\prod_{n=1}^{\infty}\frac{1}{1+q^{n^2}}.\]
Changing $ q\mapsto -q $ gives 
\begin{align*}
	H_{2}(-1;-q)&=\prod_{n=1}^{\infty}\frac{1}{1+(-q)^{n^2}}=\prod_{n=1}^{\infty}\frac{1}{\left( 1+q^{4n^2}\right) \left( 1-q^{(2n+1)^2}\right) }\\&=\prod_{n=1}^{\infty}\frac{\big( 1-q^{4n^2}\big) ^2}{\left( 1-q^{8n^2}\right) \left( 1-q^{n^2}\right) }.\end{align*}
Therefore, by setting
\[G(q)=H_{2,0,2}(-q)-H_{2,1,2}(-q),\]
we obtain 
\begin{equation*}
	G(q)=\prod_{n=1}^{\infty}\frac{\big(1-q^{4n^2}\big)^2}{\big(1-q^{8n^2}\big)\big(1-q^{n^2}\big)}=\sum_{n=0}^{\infty}a_2(n)q^n
\end{equation*}
and we want to prove that the coefficients $ a_2(n) $ are positive as $ n\to\infty. $ We will come back to this in the next section.

\subsection{Meinardus' asymptotics} Our approach is to some extent similar to that taken by Meinardus \cite{Mein} in proving his famous theorem on asymptotics of certain infinite product generating functions and described by Andrews in more detail in \cite[Ch. 6]{And}. Our case is however slightly different and, whilst we can follow some of the steps, we cannot apply his result directly and we need to make certain modifications. One of them pertains to an application of the \textit{circle method.}
\par Under certain conditions on which we do not insist for the moment, as we shall formulate similar assumptions in the course of our proof, Meinardus gives an asymptotic formula for the coefficients $ r(n) $ of the infinite product
\begin{equation}
	\label{rnmeinardus}
	f(\tau)=\prod_{n=1}^{\infty}(1-q^n)^{-a_n}=1+\sum_{r=1}^{\infty}r(n)q^n,
\end{equation} 
where $ a_n\ge 0 $ and $ q=e^{-\tau} $ with $ {\rm Re}(\tau)>0. $ 
\begin{Thm}[{Andrews \cite[Ch. 6]{And}, cf. Meinardus \cite{Mein}}]\setcounter{Thm}{0} 
	\label{Meinardus}
	As $ n\to\infty, $ we have
	\[r(n)=Cn^{\kappa}\exp\left( n^{\frac{\alpha}{\alpha+1}}\left(1+\frac{1}{\alpha}\right) ( A\Gamma(\alpha+1)\zeta(\alpha+1))^{\frac{1}{\alpha+1}}    \right)(1+O(n^{-\kappa_1}) ),  \]
	where 
	\begin{align*}
		C & =  e^{D'(0)}\left(  2\pi(1+\alpha)\right) ^{-\frac12}\left(A\Gamma(\alpha+1)\zeta(\alpha+1) \right)^{\frac{1-2D(0)}{2+2\alpha}},\\
		\kappa & =  \frac{D(0)-1-\frac12\alpha}{1+\alpha},\\
		\kappa_1 & =  \frac{\alpha}{\alpha+1}\min \left\lbrace \frac{C_0}{\alpha}-\frac{\delta}{4},\frac12-\delta \right\rbrace,  
	\end{align*}
	with $ \delta>0 $ arbitrary.
\end{Thm}

Here the Dirichlet series $$ D(s)=\sum_{n=1}^{\infty}\frac{a_n}{n^s}\quad(s=\sigma
+it) $$ is assumed to converge for $ \sigma>\alpha>0 $ and to possess an analytic continuation in the region $ \sigma>-c_0~(0<c_0<1). $ In this region $ D(s) $ is further assumed to be analytic except for a simple pole at $ s=\alpha $ with residue $ A. $
\subsection{Circle method} We now turn attention to our problem. Let $ \tau=y-2\pi ix $ and $ q=e^{-\tau}, $ with $ y>0 $ (so that $ {\rm Re}(\tau)>0 $ and $ |q|<1 $). Recall that, as defined in Section \ref{reformulation}, \begin{equation}\label{GfromH}
	G(q)=\sum_{n=0}^{\infty}a_2(n)q^n=\prod_{n=1}^{\infty}\frac{\big(1-q^{4n^2}\big)^2}{\big(1-q^{n^2}\big)\big(1-q^{8n^2}\big)}.\end{equation}
\par As one can easily see, unlike the product in \eqref{rnmeinardus}, where all factors appear to non-positive powers, the factors $ \big( 1-q^{4n^2}\big)  $ have positive exponents in the product from the right-hand side of \eqref{GfromH}. Therefore  we cannot directly apply Theorem \ref{Meinardus} to obtain asymptotics for the coefficients $ a_2(n). $  
We will, nevertheless, follow certain steps from the proof of Meinardus \cite{Mein}.  \par Let $ s=\sigma+it $ and 
\[D(s)=\sum_{n=1}^{\infty}\frac{1}{n^{2s}}+\sum_{n=1}^{\infty}\frac{1}{(8n^2)^s}-2\sum_{n=1}^{\infty}\frac{1}{(4n^2)^s}=(1+8^{-s}-2^{1-2s})\zeta(2s),\]
which is convergent for $ \sigma>\frac12=\alpha, $ has a meromorphic continuation to $ \mathbb C $ (thus we may choose $ 0<c_0<1 $ arbitrarily) and a simple pole at $ s=\frac12 $ with residue $ A=\frac{1}{4\sqrt2} .$ 
We have 
\begin{align*}
	D(0) & =   0,\\
	D'(0) & =   \zeta(0)(-3\log2+4\log2)=-\frac{\log2}{2}.
\end{align*}	
\par By Cauchy's Theorem we have, for $ n>0, $	
\[a_2(n)=\frac1{2\pi i}\int_{\mathcal C}\frac{G(q)}{q^{n+1}}dq =e^{ny}\int_{-\frac12}^{\frac12}G(e^{-y+2\pi ix})e^{-2\pi inx} dx, \]
where $ \mathcal C $ is taken to be the positively oriented circle of radius $ e^{-y} $ around the origin.
\par We choose 
\begin{equation}\label{choicey}
	y=n^{-\frac23}\left( \frac{\sqrt{\pi}}{8\sqrt2}\zeta\left(\frac32 \right)  \right)^{\frac23}>0 \end{equation}
and set 
\[m=n^{\frac13} \left( \frac{\sqrt{\pi}}{8\sqrt2}\zeta\left(\frac32 \right)  \right)^{\frac23},\]
so that $ ny=m. $ The reason for this choice of $ y $ is motivated by the \textit{saddle-point method}, which was also employed by Meinardus \cite{Mein}, and will become apparent later in the proof.
\par Moreover, let 
\begin{equation}\label{defbeta}
	\beta=1+\frac{\alpha}{2}\left( 1-\frac{\delta}{2}\right) ,\quad\text{with~} 0<\delta<\frac23,\end{equation}
so that 
\begin{equation}
	\label{ineqbeta}
	\frac76<\beta<\frac54. 
\end{equation}
\par We can then rewrite 
\begin{equation}\label{epxressionfora2}
	a_2(n)=e^{ny}\int_{-y^{\beta}}^{y^{\beta}}G(e^{-y+2\pi ix})e^{-2\pi inx}dx +R(n),
\end{equation}
where 
\[
R(n)=e^{ny} \int_{y^{\beta}\le|x|\le\frac12}  G\left( e^{-y+2\pi ix}\right) e^{-2\pi inx}dx. \]
The idea is that the main contribution for $ a_2(n) $ will be given by the integral from \eqref{epxressionfora2}, while $ R(n) $ will go into an error term. We first prove the following estimate.

\begin{Lem} 
	\label{smallxmaintermG}
	If $ |x|\le \frac12 $ and $ |\Arg(\tau)|\le\frac{\pi}{4}, $ 
	then 
	\[G\left( e^{-\tau}\right) =\frac{1}{\sqrt2}\exp\left( {\frac{\sqrt{\pi}\zeta\left( \frac32\right) }{4\sqrt2\sqrt{\tau}}}+O(y^{c_0})\right) \]
	holds uniformly in $ x $ as $ y\to0, $ with $ 0<c_0<1. $ 
\end{Lem}
\begin{proof}
	We have \[\Log  G\left( e^{-\tau}\right) =\sum_{k=1}^{\infty}\frac 1k\sum_{n=1}^{\infty}\Big( e^{-kn^2\tau}+e^{-8kn^2\tau}-2e^{-4kn^2\tau}\Big) .\]
	Using the Mellin inversion formula (see, e.g., \cite[p. 54]{Apostol2}) we get 
	\[e^{-\tau}=\frac{1}{2\pi i}\int_{\sigma_0-i\infty}^{\sigma_0+i\infty}\tau^{-s}\Gamma(s)ds\] 
	for $ {\rm Re}(\tau)>0$ and $\sigma_0>0, $  
	thus
	\begin{align}\label{integrand}
		\Log  G\left( e^{-\tau}\right) &=\frac{1}{2\pi i}\int_{1+\alpha-i\infty}^{1+\alpha+i\infty}\Gamma(s)\sum_{k=1}^{\infty}\frac 1k\sum_{n=1}^{\infty}( (kn^2\tau)^{-s}+(8kn^2\tau)^{-s}-2(4kn^2\tau)^{-s} )ds\nonumber\\ 
		&=\frac{1}{2\pi i}\int_{\frac32-i\infty}^{\frac32+i\infty}\Gamma(s)D(s)\zeta(s+1)\tau^{-s}ds.
	\end{align}
	By assumption, \[|\tau^{-s}|=|\tau|^{-\sigma}e^{t\cdot\Arg( \tau)}\le|\tau|^{-\sigma}e^{\frac{\pi}{4}|t|}.\]
	Classical results (see, e.g., \cite[Ch. 1]{AAR} and \cite[Ch. 5]{Titch}) tell us that the bounds
	\begin{align*}
		D(s) &= O(|t|^{c_1}),\\
		\zeta(s+1) &= O(|t|^{c_2}),\\
		\Gamma(s)&= O\Big( e^{-\frac{\pi|t|}{2}}|t|^{c_3} \Big)
	\end{align*}
	hold uniformly in $ -c_0\le\sigma\le\frac32 =1+\alpha$  as $ |t|\to\infty, $ for some $ c_1,c_2 $ and $ c_3>0. $
	\par Thus we may shift the path of integration to $ \sigma=-c_0. $ The integrand in \eqref{integrand} has poles at $ s=\frac12 $ and $ s=0,$  with residues
	\begin{eqnarray*}
		&& \Res_{s=\frac12}\left( \Gamma(s)D(s)\zeta(s+1)\tau^{-s}\right) =\Gamma\left(\dfrac12 \right) A\zeta\left(\dfrac32 \right) \tau^{-\frac12},\\
		&& \Res_{s=0}\left( \left( \frac 1s+O(1)\right)(D'(0)s+O(s^2) )  \left( \frac1s+O(1) \right)   \left( 1+O(s)\right) \right)=D'(0)=-\frac{\log2}{2}. 
	\end{eqnarray*}
	The remaining integral equals 
	\begin{align*}
		\frac{1}{2\pi i}\int_{-c_0-i\infty}^{-c_0+i\infty}\tau^{-s}\Gamma(s)D(s)\zeta(s+1)ds&\ll|\tau|^{c_0}\int_{0}^{\infty}t^{c_1+c_2+c_3}e^{-\frac{\pi t}{4}}dt\\&\ll |\tau|^{c_0}=|y-2\pi ix|^{c_0}\\&\le( \sqrt2y) ^{c_0}
	\end{align*}
	since, again by the assumption, \[\frac{2\pi |x|}{y}=\tan (|\Arg(\tau)|)\le \tan\left(\frac{\pi}{4} \right) =1.\]
	We therefore obtain
	\[\Log  G\left( e^{-\tau}\right) =\left( \frac{\zeta\left(\frac32 \right)\sqrt{\pi} }{4\sqrt2\sqrt{\tau}}-\frac{\log2}{2}\right) +O(y^{c_0}),\] which completes the proof.
\end{proof}

The proof of the upcoming Lemma \ref{boundGbigoh} is similar in spirit with that of part (b) of the \textit{Hilfssatz} (Lemma) in Meinardus \cite[p. 390]{Mein} or, what is equivalent, the second part of Lemma 6.1 in Andrews \cite[Ch. 6]{And}. Our case is however more subtle, in that it involves some extra factors $ P_{a,b} $ (which will be explained below) and requires certain modifications. For this we need a setup in which to apply the circle method as described by Wright \cite[p. 172]{WrightIII}. For a nice introduction to the circle method and the theory of Farey fractions, the reader is referred to Apostol \cite[Ch. 5.4]{Apostol}.
\par We consider the Farey dissection of order $ \left\lfloor y^{-\frac23} \right\rfloor $ of $ \mathcal C $ and distinguish two kinds of arcs:
\begin{itemize}
	\item major arcs,  denoted $ \mathfrak M_{a,b}, $ such that $ b\le y^{-\frac13}; $
	\item minor arcs, denoted $ \mathfrak m_{a,b}, $ such that $ y^{-\frac13}<b\le y^{-\frac23}. $
\end{itemize}
\par We write any $ \tau\in \mathfrak M_{a,b}\cup \mathfrak m_{a,b} $ as 
\begin{equation}\label{tautauprime}
	\tau=y-2\pi ix=\tau'-2\pi i\frac ab
\end{equation}
with $ \tau'=y-2\pi ix'. $ From basics of Farey theory (alternatively, see Wright \cite[p. 172]{WrightIII}) it follows that 
\begin{equation}
	\label{ineqxprim}
	\frac{y^{\frac23}}{2b}\le |x'|\le\frac{y^{\frac23}}{b}. 
\end{equation} 
\subsection{Wright's modular transformations}
Our next step requires us to apply the modular transformations found by Wright \cite{WrightIII} for the generating functions of partitions into $ r $-th powers. In what follows, we choose the principal branch of the square root.  
In the notation introduced in the previous subsection, 
the modular transformation law obtained by Wright \cite[Theorem 4]{WrightIII} rewrites as
\begin{equation}\label{Wrightexp}
	H_2(q)=H_2\left( e^{\frac{2\pi ia}{b}-\tau'} \right)=C_{b}\sqrt{\tau'}\exp\left( {\frac{\Lambda_{a,b}}{\sqrt{\tau'}}}\right)P_{a,b}(\tau'),  \end{equation}
where  
\begin{equation}
	\label{definitionLambda}
	\Lambda_{a,b}=\frac{\Gamma\left( \frac32\right)} {b}\sum_{m=1}^{\infty}\frac{S_{ma,b}}{m^{\frac32}},
\end{equation}
\begin{equation}
	\label{definitionsumS}
	S_{a,b}=\sum_{n=1}^b \exp\left( \frac{2\pi ia n^2}{b}\right),
\end{equation}
and \[C_{b}={\frac{b_1}{2\pi},}\]
with $ 0\le a<b $ coprime integers and $ b_1 $ the least positive integer such that $ b\mid b_1^2 $ and $ b=b_1b_2, $ 
\[P_{a,b}(\tau')=\prod_{h=1}^{b}\prod_{s=1}^2\prod_{\ell=0}^{\infty}\left( 1-g(h,\ell,s)\right)^{-1}, \]
with 
\[g(h,\ell,s)=\exp\left(\frac{(2\pi)^{\frac 32}(\ell+\mu_{h,s})^{\frac 12}e^{\frac{\pi i}{4}(2s+1)}}{b\sqrt{\tau'}}-\frac{2\pi ih}{b} \right),  \]
where  $ 0\le d_h<b $ is defined by the congruence \[ah^2\equiv d_h\pmod*{b}\] and
\[\mu_{h,s}=\begin{cases}
\frac{d_h}{b} &\text{if~}s=1,\\
\frac{b-d_h}{b}&\text{if~}s=2,
\end{cases}\]	
for $ d_h\ne0. $ If $ d_h=0, $ we let $ \mu_{h,s}=1. $ 
\par Our goal is to establish the following result, the proof of which we give at the end of the section.
\begin{Lem}
	\label{boundGbigoh} 
	There exists $ \varepsilon>0 $ such that, as $ y\to0, $  
	\[G\left( e^{-\tau}\right) =O\left( e^{\frac{\Lambda_{0,1}}{2\sqrt{2y}}-cy^{-\varepsilon}} \right) \]
	holds uniformly in $ x $ with $ y^{\beta}\le |x|\le\frac 12, $ for some $ c>0. $ \end{Lem}

Recall that $ q=e^{-\tau}, $ with $ y>0 $ (so that $ {\rm Re}(\tau)>0 $ and $ |q|<1 $). From \eqref{GfromH}, \eqref{tautauprime} and \eqref{Wrightexp} we have, for some positive constant $ C $ that can be made explicit, 
\begin{equation}\label{formforG}
	G(q)=\frac{H(q)H(q^8)}{H(q^4)^2}=C\exp\left( \frac{\lambda_{a,b}}{\sqrt{\tau'}}\right) \frac{P_{a,b}(\tau')P_{a,b}'(8\tau')}{P_{a,b}''(4\tau')^2},
\end{equation}
where
\[P'_{a,b}=P_{\frac{8a}{(b,8)},\frac{b}{(b,8)}},\quad P_{a,b}''=P_{\frac{4a}{(b,4)},\frac{b}{(b,4)}}\]
and 
\begin{equation}
	\label{lambda_ab}
	\lambda_{a,b}=\Lambda_{a,b}+\frac{1}{2\sqrt2} \Lambda_{\frac{8a}{(b,8)},\frac{b}{(b,8)}}-\Lambda_{\frac{4a}{(b,4)},\frac{b}{(b,4)}}.
\end{equation}
Additionally, set 
\begin{equation}\label{stars}
	\Lambda^*_{a,b}=\frac{\Lambda_{a,b}}{\Gamma\left(\frac32 \right)}\quad\text{and}\quad \lambda^*_{a,b}=\frac{\lambda_{a,b}}{\Gamma\left(\frac32 \right) }.
\end{equation}
We want to study the behavior of $ P_{a,b}(\tau'). $ 
\begin{Lem}
	\label{majorarcs} If $ \tau\in\mathfrak M_{a,b}\cup\mathfrak m_{a,b}, $ then
	
	\[\log|P_{a,b}(\tau')|\ll b\quad\text{as $ y\to 0. $}\]
\end{Lem}
\begin{proof}
	Using \eqref{ineqxprim} and letting $ y\to0 $, we have \[|\tau'|^{\frac32}=(y^2+4\pi^2x'^2)^{\frac34}\le\left(y^2+\frac{4\pi^2y^{\frac43}}{b^2} \right)^{\frac34}\le\frac{c_4y}{b^{\frac32}}=\frac{c_4 {\rm{Re}}\left(\tau'\right) }{b^{\frac32}}, \]
	for some $ c_4>0. $ Thus, \cite[Lemma 4]{WrightIII} gives
	\[|g(h,\ell,s)|\le e^{-c_5(\ell+1)^{\frac12}},\]
	with $ c_5=\frac{2\sqrt{2\pi}}{c_4}, $ which in turn leads to
	\[| \log| P_{a,b}(\tau') || \le \sum_{h=1}^{b}\sum_{s=1}^{2}\sum_{\ell=1}^{\infty}|\log(1-g(h,\ell,s))|\le 2b\sum_{\ell=1}^{\infty}\Big| \log\Big( 1-e^{-c_5(\ell+1)^{\frac12}}\Big)\Big| \ll b,\] concluding the proof.
\end{proof}
\subsection{Final lemmas} We first want to bound $ G(q) $ on the minor arcs. 
\begin{Lem}
	\label{minors}
	If $ \varepsilon>0 $ and $ \tau\in\mathfrak m_{a,b}, $ 
	then
	\begin{equation*}\label{correctedminorarcs}
		|\Log G(q)|\ll_{\varepsilon}{y^{\frac16-\varepsilon}}.
	\end{equation*}  
\end{Lem}
\begin{proof}
	In the proof and notation of \cite[Lemma 17]{WrightIII}, replace $ a=\frac{1}{2}, $ $ b=\frac13, $ $ c=2, $ $ \gamma=\varepsilon $ and $ N=y^{-1}. $
\end{proof}
Before delving into the proof of Lemma \ref{boundGbigoh} we need two final, though tedious, steps. 
\begin{Lem}
	\label{boundmaxReIm}
	If $0\le a<b$ are coprime integers with $ b\ge2, $ we have \[\max \left\lbrace \left|  {\rm Re}\left( \lambda_{a,b}\right)\right| ,\left| {\rm Im} \left( \lambda_{a,b}\right)\right|\right\rbrace <\frac{\zeta\left( \frac32\right)\Gamma\left( \frac32\right)  }{1.14\cdot2\sqrt2}.  \]
\end{Lem}
\begin{proof}
	A well-known result due to Gauss (for a proof see, e.g., \cite[Ch. 1]{Berndt}) says that, for $ (a,b)=1, $ the sum $S_{a,b} $ defined in \eqref{definitionsumS} can be computed by the formula
	\[S_{a,b}=\begin{cases}
	0 & \text{if~}b\equiv2\pmod*4,\\
	\varepsilon_b\sqrt b\left( \frac ab \right) & \text{if~$2\nmid b,$}\\
	(1+i)\varepsilon_a^{-1}\sqrt b\left( \frac ba \right) & \text{if~}4\mid b,  
	\end{cases}\] where $ \left( \frac ab\right)  $ is the usual Jacobi symbol and 
	\[\varepsilon_b=\begin{cases}
	1 & \text{if~} b\equiv1\pmod*4,\\
	i & \text{if~} b\equiv3\pmod*4.
	\end{cases}\]
	On recalling \eqref{definitionLambda}, \eqref{lambda_ab} and \eqref{stars}, it is enough to prove that \[\max \{ |  {\rm Re}( \lambda_{a,b}^*)| ,| {\rm Im} ( \lambda_{a,b}^*)|\} <\frac{\zeta\left( \frac32\right) }{1.14\cdot2\sqrt2}.  \]
	We explicitly evaluate $ \Lambda^*_{a,b}. $
	We have, on using the fact that $ S_{ma,b}=d S_{ma/d,b/d} $ to prove the second equality below, and on replacing $ m\mapsto md $ and $ d\mapsto \frac bd $ to prove the third and fourth respectively,   
	\begin{align*}
		\Lambda^*_{a,b}&=\frac 1b\sum_{m=1}^{\infty}\frac{S_{ma,b}}{m^{\frac32}}=\frac 1b\sum_{d\mid b}\sum_{\substack{m\ge1\\(m,b)=d}}\frac{d S_{ma/d, b/d}}{m^{\frac32}}=\frac 1b\sum_{d|b}d\sum_{\substack{m\ge1\\(m,b/d)=1}}\frac{S_{m a, b/d}}{(m d)^{\frac32}}\\
		&=\frac 1b\sum_{d\mid b}d^{-\frac12}\sum_{\substack{m\ge1\\(m,b/d)=1}}\frac{S_{m a, b/d}}{m^{\frac32}}=\frac 1b\sum_{d|b}\left( \frac bd\right)^{-\frac12} \sum_{\substack{m\ge1\\(m,d)=1}}\frac{S_{ma,d}}{m^{\frac32}}\\&=\frac{1}{b^{\frac32}}\sum_{d|b}d^{\frac12}\sum_{\substack{m\ge1\\(m,d)=1}}\frac{S_{ma,d}}{m^{\frac32}}.
	\end{align*}
	We  distinguish several cases, in all of which we shall apply the following bound for divisor sums. 
	If $ \beta, L, \ell\in\bb N $  and $ \gamma $ is the \textit{Euler-Mascheroni} constant, then
	\begin{align}
		\label{boundcase1}
		\sum_{\substack{d|\beta\\d\equiv\ell\pmod*L}}\frac1d&\le \sum_{\substack{1\le Ld+\ell\le \beta\\0\le d\le \frac{\beta-\ell}{L}}}\frac{1}{Ld+\ell}\le\frac{1}{\ell}+\frac{1}{L}\sum_{1\le d\le \frac{\beta}{L}}\frac1d\nonumber\\&\le\frac{1}{\ell}+\frac1L\left( \log\left( \frac{\beta}{L}\right)+\gamma+\frac{1}{\frac{2\beta}{L}+\frac13} \right). 
	\end{align}
	\begin{Rem} The first inequality in \eqref{boundcase1} can be easily deduced, while the second one was posed as a problem in the \textit{American Mathematical Monthly} by T\'oth \cite[Problem E3432]{AMM} and can be solved by usual techniques like summation by parts and integral estimates.
		
	\end{Rem}
	\noindent{\sc Case 1:} $ 2\nmid b. $ We have \begin{align*}
		\lambda_{a,b}^*&=\Lambda^*_{a,b}+\frac{1}{2\sqrt2}\Lambda^*_{8a,b}-\Lambda^*_{4a,b}\\&=\frac{1}{b^{\frac32}}\sum_{d|b}d^{\frac12}\sum_{\substack{m\ge1\\(m,d)=1}}\frac{1}{m^{\frac32}}\left( S_{ma,d}+\frac{S_{8ma,d}}{2\sqrt2}-S_{4ma,d}\right) \\ &= \frac{1}{2\sqrt2b^{\frac32}}\sum_{d|b}d\varepsilon_d\left(\frac{2a}{d} \right)\sum_{\substack{m\ge1\\(m,d)=1}}\frac{\left(\frac md \right) }{m^{\frac32}}. \end{align*}
	In case $ b\equiv1\pmod*4 $ we bound both the real and imaginary part of $ \lambda_{a,b}^* $ (for $ j=1,3 $ respectively) by 
	\begin{align*}
		\frac{1}{2\sqrt2b^{\frac32}}\sum_{\substack{d|b\\d\equiv j\pmod*4}}d\zeta\left(\frac32\right)&=\frac{1}{2\sqrt2b^{\frac32}}\sum_{\substack{d|b\\d\equiv j\pmod*4}}\frac bd\zeta\left(\frac32\right)\\&=\frac{1}{2\sqrt2b^{\frac12}}\sum_{\substack{d|b\\d\equiv j\pmod*4}}\frac1d\zeta\left(\frac32\right), 
	\end{align*}
	whilst for $ b\equiv3\pmod*4 $ we can bound the two quantities by
	\begin{align*}
		\frac{1}{2\sqrt2b^{\frac32}}\sum_{\substack{d|b\\d\equiv j\pmod*4}}d\zeta\left(\frac32\right)&=\frac{1}{2\sqrt2b^{\frac32}}\sum_{\substack{d|b\\d\equiv j+2\pmod*4}}\frac bd\zeta\left(\frac32\right)\\&=\frac{1}{2\sqrt2b^{\frac12}}\sum_{\substack{d|b\\d\equiv j+2\pmod*4}}\frac1d\zeta\left(\frac32\right). 
	\end{align*}
	Using the bound \eqref{boundcase1} in the worst possible case (that is, $ d\equiv1\pmod*4 $) gives 
	\[\sum_{\substack{d|b\\d\equiv1\pmod* 4}}\frac 1d\le  1+\frac14\left(\log\left(\frac b4 \right)+\gamma+\frac{1}{\frac b2+\frac13}  \right). \]
	We checked in MAPLE that 
	\[\frac{\zeta\left( \frac32\right) }{2\sqrt2 b^{\frac12}} \left( 1+\frac14\left(\log\left(\frac b4 \right)+\gamma+\frac{1}{\frac b2+\frac13}  \right) \right)<\frac{\zeta\left( \frac32\right) }{1.14\cdot2\sqrt2 }\]
	for $ b>1. $ Since the left-hand side above is a decreasing function,
	we are done in this case.\\
	
	\noindent{\sc Case 2:} $ 2\parallel b. $ As $ S_{a,b}=0 $ for $ b\equiv2\pmod*4, $ we have
	\begin{align*}
		\lambda_{a,b}^*&=\Lambda^*_{a,b}+\frac{1}{2\sqrt2}\Lambda^*_{4a,\frac b2}-\Lambda^*_{2a,\frac b2}\\&=\frac{1}{b^{\frac32}} \sum_{d\left| \frac b2\right. }d^{\frac12}\sum_{\substack{m\ge1\\(m,d)=1}}\frac{1}{m^{\frac32}} \left(  S_{ma,d}+S_{4ma,d}-2\sqrt2 S_{2ma,d}\right)  \\
		&=\frac{2}{b^{\frac32}}\sum_{d\left| \frac b2\right. }d^{\frac12}\sum_{\substack{m\ge1\\(m,d)=1}}\frac{\varepsilon_d\left( \frac{ma}{d}\right)\left( 1-\sqrt2\left(\frac 2d \right) \right)\sqrt{d}}{m^{\frac32}} \\ 
		&=\frac{2}{b^{\frac32}}\sum_{d\left| \frac b2\right. }d\left( \frac ad\right)\varepsilon_d\left( 1-\sqrt2\left(\frac 2d \right) \right) \sum_{\substack{m\ge1\\(m,d)=1}}\frac{\left( \frac{m}{d}\right)}{m^{\frac32}}.
	\end{align*}
	Taking real and imaginary parts gives (for $ j=1,3 $ respectively, and some $ \ell=1,3 $ depending on the congruence class of $ \frac b2\pmod*8 $)
	\begin{align*}
		&\phantom{<~}\frac{2}{b^{\frac32}} \sum_{\substack{d\left| \frac b2\right.\\d\equiv j\pmod*4 }}d\left(\frac ad \right)\left( 1-\sqrt2\left(\frac 2d \right)\right)  \sum_{\substack{m\ge1\\(m,d)=1}}\frac{\left(\frac md \right) }{m^{\frac32}}\\&\le\frac{\zeta\left( \frac32\right) }{b^{\frac12}}\left( \sum_{\substack{d\left| \frac b2\right.\\d\equiv \ell\pmod*8 } }\frac1d(\sqrt2-1) + \sum_{\substack{d\left| \frac b2\right.\\d\equiv \ell+4\pmod*8 } }\frac1d(\sqrt2+1)\right).
	\end{align*}
	We now use \eqref{boundcase1} in the worst possible case (that is, $ \ell+4\equiv1\pmod*8 $) to obtain the bound
	\begin{align*}&\phantom{+~~}\frac{\zeta\left( \frac32\right) }{b^{\frac12}}\left( ( \sqrt2-1)\left(\frac15+\frac18\left( \log\left(\frac{b}{16}\right) +\gamma+\frac{1}{\frac b8+\frac13} \right) \right)\right. \\  & + \left.(\sqrt2+1)\left
		(1+\frac18\left( \log\left(\frac{b}{16}\right) +\gamma+\frac{1}{\frac b8+\frac13}  \right)  \right)  \right) .
	\end{align*}
	This is a decreasing function and a computer check in MAPLE shows that it is bounded above by $\frac{\zeta\left( \frac32\right) }{1.14\cdot2\sqrt2} $
	for $ b\ge 124. $ For the remaining cases we use the well-known relation between a Dirichlet $ L $-series and the Hurwitz zeta function (see, e.g., Apostol \cite[Ch. 12]{Apostol}) to write 
	\[\lambda_{a,b}^*=\frac{2}{b^{\frac32}}\sum_{d\left| \frac b2\right.}d^{-\frac12}\varepsilon_d\left( 1-\sqrt2\left(\frac 2d \right)\right) \sum_{\ell=1}^{d}\left( \frac{\ell a }{d}\right)\zeta\left( \frac32,\frac{\ell}{d} \right).  \]
	For $ b\le 124 $ we checked in MAPLE that
	\[\max \{ |  {\rm Re}( \lambda_{a,b}^*)| ,| {\rm Im} ( \lambda_{a,b}^*)|\} <\frac{\zeta\left( \frac32\right) }{1.14\cdot2\sqrt2}.  \]
	\noindent{\sc Case 3:} $ 4\parallel b. $ We have
	\begin{align*}
		\lambda_{a,b}^*&=\Lambda^*_{a,b}+\frac{1}{2\sqrt2}\Lambda^*_{2a,\frac b4}-\Lambda^*_{a,\frac b4}\\&=\frac{1}{b^{\frac32}} \sum_{d|b}d^{\frac12}\sum_{\substack{m\ge1\\(m,d)=1}}\frac{S_{ma,d}}{m^{\frac32}} +\frac{8}{b^{\frac32}}\sum_{d\left| \frac b4\right. }d^{\frac12}\sum_{\substack{m\ge1\\(m,d)=1}}\frac{1}{m^{\frac32}}\left( \frac{S_{2ma,d}}{2\sqrt2}-S_{ma,d} \right)\\&= \frac{1}{b^{\frac32}}\sum_{d\left| \frac b4\right. }d^{\frac12}\sum_{\substack{m\ge1\\(m,d)=1}}\frac{1}{m^{\frac32}}\left( S_{ma,d}+2\sqrt2 S_{2ma,d}-8 S_{ma,d} \right) \\&\phantom{=~}+\frac{1}{b^{\frac32}}\sum_{d\left| \frac b4\right. }(4d)^{\frac12}\sum_{\substack{m\ge1\\(m,2d)=1}}\frac{S_{ma,4d}}{m^{\frac32}}
		\\&= \frac{1}{b^{\frac32}}\sum_{d\left| \frac b4\right. }d^{\frac12}\sum_{\substack{m\ge1\\(m,d)=1}}\frac{\varepsilon_d\sqrt d\left( \frac{ma}{d}\right)\left( -7+2\sqrt2\left(\frac 2d \right) \right)}{m^{\frac32}}\\&\phantom{=~}+\frac{1}{b^{\frac32}}\sum_{d\left| \frac b4\right. }(4d)^{\frac12}\sum_{\substack{m\ge1\\(m,2d)=1}}\frac{(1+i)\varepsilon^{-1}_{ma}2\sqrt d\left(\frac{4d}{ma} \right)}{m^{\frac32}}.
	\end{align*}
	In the same way as before, the real and imaginary parts of $ \lambda_{a,b}^* $ can be bounded (for some $ j=1,3 $ depending on the congruence class of $ \frac b4\pmod*4 $) by
	\begin{align*}
		&\phantom{=~}\frac{\zeta\left( \frac32\right) }{b^{\frac32}} \left( \sum_{\substack{d\left| \frac b4\right.\\d\equiv j\pmod*8 } }d(7+2\sqrt2) + \sum_{\substack{d\left| \frac b4\right.
				\\d\equiv j+4\pmod*8 } }d(7-2\sqrt2)+ 4\left( 1-2^{-\frac32}\right)  \sum_{d\left| \frac b4\right.}d\right) \\
		&=\frac{\zeta\left( \frac32\right) }{b^{\frac32}} \left( \sum_{\substack{d\left| \frac b4\right.\\d\equiv j\pmod*8 } }\frac{b}{4d}(7+2\sqrt2) + \sum_{\substack{d\left| \frac b4\right.
				\\d\equiv j+4\pmod*8 } }\frac{b}{4d}(7-2\sqrt2)+ \left( 1-2^{-\frac32}\right)  \sum_{d\left| \frac b4\right. }\frac{b}{d}\right) ,
	\end{align*}
	which, by using \eqref{boundcase1} in the worst possible case (that is, $ j+4\equiv5\pmod*8 $), is seen to be less than 
	\begin{align*}
		&\phantom{+~}\frac{\zeta\left( \frac32\right) }{4b^{\frac12}} \Bigg( ( 7+2\sqrt2)\left(1+\frac18\left( \log\left(\frac{b}{32}\right) +\gamma+\frac{1}{\frac{b}{16}+\frac13} \right) \right) \\&+ (7-2\sqrt2)\left
		(\frac15+\frac18\left( \log\left(\frac{b}{32}\right) +\gamma+\frac{1}{\frac{b}{16}+\frac13}  \right)  \right)   \\
		& +4\left( 1-2^{-\frac32}\right)\left( 1+ \frac12\left(\log\left( \frac b8\right) +\gamma+\frac{1}{\frac b4+\frac13} \right)\right) \Bigg) .
	\end{align*}
	In turn, a computer check in MAPLE shows that this decreasing function is bounded above by $\frac{\zeta\left( \frac32\right) }{1.14\cdot2\sqrt2} $ for $ b\ge 390. $ 
	For the remaining cases we rewrite
	\begin{align*}
		\lambda_{a,b}^*&=\frac{1}{b^{\frac32}}\sum_{d\left| \frac b4\right.}\left( d^{-\frac12}\varepsilon_d\left( -7+2\sqrt2\left(\frac 2d \right) \right) 
		\sum_{\ell=1}^d\left( \frac{\ell a}{d}\right)\zeta\left(\frac32,\frac{\ell}{d} \right)
		\right. \\&\phantom{=~}+ \left.(4d)^{-\frac12}(1+i)\sum_{\ell=1}^{4d}\varepsilon_{\ell a}^{-1}\left( \frac{4d}{\ell a}\right)\zeta\left(\frac32,\frac{\ell}{4d} \right) \right)  .   
	\end{align*}
	A MAPLE check shows that, for $ b\le 390, $ we have
	\[\max \{ |  {\rm Re}( \lambda_{a,b}^*)| ,| {\rm Im} ( \lambda_{a,b}^*)|\} <\frac{\zeta\left( \frac32\right) }{1.14\cdot2\sqrt2}.  \]
	\noindent{\sc Case 4:} $ 8\mid b. $ We write $ b=2^{\nu}b', $ with $ b'  $ odd. If we define $ \delta_{d,4}=0 $ for $ 4\nmid d $ and $ \delta_{d,4}=1 $ for $ 4\mid d, $ we have 
	\begin{align*}
		\lambda_{a,b}^*&=\Lambda^*_{a,b}+\frac{\Lambda^*_{a,\frac b8}}{2\sqrt2}-\Lambda^*_{a,\frac b4}\\&=\frac{1}{b^{\frac32}} \sum_{d|b}d^{\frac12}\sum_{\substack{m\ge1\\(m,d)=1}}\frac{1}{m^{\frac32}} \left( \varepsilon_d\left(\frac{4ma}{d} \right)\sqrt d+\delta_{d,4}\varepsilon^{-1}_{ma}(1+i)\sqrt d \left( \frac{d}{ma} \right)   \right)   \\&\phantom{=~}+ \frac{1}{\left(\frac b8 \right) ^{\frac32}2\sqrt2}\sum_{d\left| \frac b8\right. }d^{\frac12}\sum_{\substack{m\ge1\\(m,d)=1}} \frac{1}{m^{\frac32}}\left( \varepsilon_d\left(\frac{4ma}{d} \right)\sqrt d+\delta_{d,4}\varepsilon^{-1}_{ma}(1+i)\sqrt d \left( \frac{d}{ma} \right)   \right) \\&\phantom{=~}-\frac{1}{\left(\frac b4 \right) ^{\frac32}}\sum_{d\left| \frac b4\right. }d^{\frac12}\sum_{\substack{m\ge1\\(m,d)=1}} \frac{1}{m^{\frac32}}\left( \varepsilon_d\left(\frac{4ma}{d} \right)\sqrt d+\delta_{d,4}\varepsilon^{-1}_{ma}(1+i)\sqrt d \left( \frac{d}{ma} \right)   \right) \\&=\frac{1}{b^{\frac32}}\sum_{d|b'}d\sum_{\substack{m\ge1\\(m,d)=1}}\frac{\varepsilon_d\left( \frac{ma}{d}\right) }{m^{\frac32}}+\frac{1+i}{b^{\frac32}}\sum_{\substack{d|b'\\2\le j\le \nu-3}}d\cdot 2^j\sum_{\substack{m\ge1\\(m,2d)=1}}\frac{\varepsilon_{ma}^{-1}\left( \frac{2^jd}{ma}\right) }{m^{\frac32}}\\&\phantom{=~}-\frac{7(i+1)}{b^{\frac32}}\sum_{d|b'}d\cdot2^{\nu-2}\sum_{\substack{m\ge1\\(m,2d)=1}}\frac{\varepsilon^{-1}_{ma}\left( \frac{2^{\nu-2}d}{ma}\right) }{m^{\frac32}}\\&\phantom{=~}+\frac{1+i}{b^{\frac32}}\sum_{\substack{d|b'\\ \nu-1\le j\le \nu}}d\cdot 2^j\sum_{\substack{m\ge1\\(m,2d)=1}}\frac{\varepsilon_{ma}^{-1}\left( \frac{2^jd}{ma}\right) }{m^{\frac32}}.
	\end{align*}
	Taking real and imaginary parts gives, for $ \ell,k\in\{1,3
	\} $ depending on the congruence class of $ b'\pmod*4, $ 
	\begin{align*}&\phantom{~=~}\frac{\zeta\left( \frac32\right) }{b^{\frac32}}\sum_{\substack{d|b'\\d\equiv \ell\pmod*4}}d+\frac{\zeta\left( \frac32\right) \left (1-2^{-\frac32}\right) }{b^{\frac32}}\sum_{d|b'}d\left( 3\cdot 2^{\nu-1}+\sum_{2\le j\le \nu} 2^j\right)\\
		&=\frac{\zeta\left( \frac32\right) }{b^{\frac32}}\sum_{\substack{d|b'\\d\equiv k\pmod*4}}\frac b{2^{\nu}d}+\frac{\zeta\left( \frac32\right) \left( 1-2^{-\frac32}\right) }{b^{\frac32}}\sum_{d|b'}\frac{b}{2^{\nu}d}\left( 3\cdot 2^{\nu-1}+\sum_{2\le j\le \nu} 2^j\right) \end{align*}
	as bound for $\max \{ |  {\rm Re}( \lambda_{a,b}^*)| ,| {\rm Im} ( \lambda_{a,b}^*)|\}.$
	The expression inside the brackets from the inner sum equals 
	$7\cdot 2^{\nu-1}-4<7\cdot 2^{\nu-1},$ and thus we obtain  as overall bound, in the worst possible case (that is, $ d\equiv1\pmod*4 $),
	\begin{align*}
		&\phantom{~<~}\zeta\left( \frac32\right) \left( \frac{1}{b^{\frac12}2^{\nu}}\sum_{\substack{d|b'\\d\equiv1 \pmod*4}} \frac1d +\frac{7\left( 1-2^{-\frac32}\right) }{2b^{\frac12}}\sum_{\substack{d|b'\\d\equiv1\pmod*2}}\frac1d\right)\\
		& \le\frac{\zeta\left( \frac32\right) }{b^{\frac12}}\left(\frac18\left( 1+\frac14\left(\log\left( \frac{b'}{4}\right) +\gamma+\frac{1}{\frac{b'}{2}+\frac13} \right) \right)\right.  \\&\phantom{<=} \left.  +\frac72\left( 1-2^{-\frac32}\right) \left(1+\frac12\left(\log\left( \frac{b'}{2}\right) +\gamma+\frac{1}{b'+\frac13} \right)  \right)\right) \\
		&\le \frac{\zeta\left( \frac32\right) }{b^{\frac12}}\left(\frac18\left( 1+\frac14\left(\log\left( \frac{b}{32}\right) +\gamma+\frac{1}{\frac{b}{16}+\frac13} \right) \right)\right.  \\&\phantom{<=} \left. +\frac72\left( 1-2^{-\frac32}\right) \left(1+\frac12\left(\log\left( \frac{b}{16}\right) +\gamma+\frac{1}{\frac b8+\frac13} \right) \right)  \right).
	\end{align*}
	A computer check in MAPLE shows that this last expression, which is a decreasing function, is bounded above by $\frac{\zeta\left( \frac32\right) }{1.14\cdot2\sqrt2} $ for $ b\ge 527. $
	For the remaining cases we rewrite
	\begin{align*}
		\lambda_{a,b}^*&=\frac{1}{b^{\frac32}}\sum_{d|b}d\varepsilon_d\sum_{m\ge1}\frac{\left(\frac{4ma}{d} \right) }{m^{\frac32}}+\frac{1}{b^{\frac32}}\sum_{d\left| \frac b4\right. }4d(1+i)\sum_{m\ge1}\frac{\varepsilon_{ma}^{-1}\left(\frac{4d}{ma} \right) }{m^{\frac32}}\\&\phantom{=~}+\frac{8}{b^{\frac32}}\sum_{d\left| \frac {b}{32}\right. }\sum_{m\ge1}4d(1+i)\frac{\varepsilon_{ma}^{-1}\left(\frac{4d}{ma} \right) }{m^{\frac32}}\\
		&\phantom{=~}-\frac{8}{b^{\frac32}}\sum_{d\left| \frac {b}{16}\right. }4d(1+i)\sum_{m\ge1}\frac{\varepsilon_{ma}^{-1}\left(\frac{4d}{ma} \right) }{m^{\frac32}}\\
		&=\frac{1}{b^{\frac32}}\sum_{d|b}d^{-\frac12}\varepsilon_d\sum_{\ell=1}^d\left(\frac{4\ell a}{d} \right)\zeta\left( \frac32,\frac{\ell}{d}\right) +\frac{1+i}{b^{\frac32}}\sum_{d\left| \frac b4\right. }(4d)^{-\frac12}\sum_{\ell=1}^{4d}\varepsilon_{\ell a}^{-1}\left(\frac{4d}{\ell a} \right) \zeta\left( \frac32,\frac{\ell}{4d}\right)\\&\phantom{=~} +\frac{8(i+1)}{b^{\frac32}}\sum_{d\left| \frac {b}{32}\right. }(4d)^{-\frac12}\sum_{\ell=1}^{4d}\varepsilon_{\ell a}^{-1}\left(\frac{4d}{\ell a} \right) \zeta\left( \frac32,\frac{\ell}{4d}\right) \\&\phantom{=~}-\frac{8(i+1)}{b^{\frac32}}\sum_{d\left| \frac {b}{16}\right. }(4d)^{-\frac12}\varepsilon_{\ell a}^{-1}\left(\frac{4d}{\ell a} \right) \zeta\left( \frac32,\frac{\ell}{4d}\right)
	\end{align*}
	and check that \[\max \{ |  {\rm Re}( \lambda_{a,b}^*)| ,| {\rm Im} ( \lambda_{a,b}^*)|\} <\frac{\zeta\left( \frac32\right) }{1.14\cdot2\sqrt2}.  \]
	This finishes the proof of the lemma.	
\end{proof}
\begin{Lem}
	\label{biglemma}
	If $0\le a<b$ are coprime integers with $ b\ge2, $ for some $ c>0 $ we have  
	\[\frac{\lambda_{0,1}}{\sqrt y}-{\rm{Re}}\left( \frac{\lambda_{a,b}}{\sqrt{\tau'}} \right)\ge \frac{c}{\sqrt{y}}.  \]
\end{Lem}
\begin{proof}
	We write $ \tau'=y+ity  $ for some $ t\in\bb R. $ We have
	\begin{align*}
		{\rm{Re}}\left( \frac{\lambda_{a,b}}{\sqrt{\tau'}} \right)&=\frac{1}{\sqrt y}{\rm Re}\left(\frac{\lambda_{a,b}}{\sqrt{1+it }} \right)=\frac{1}{\sqrt y}{\rm Re}\left( \frac{\lambda_{a,b}}{(1+t^2)^{\frac14}e^{\frac i2\arctan t}}\right)\\
		&= \frac{1}{\sqrt y(1+t^2)^{\frac14}}\left( \cos\left(\frac{\arctan t}{2}\right) {\rm Re}\left( \lambda_{a,b}\right)+\sin  \left(\frac{\arctan t}{2}\right){\rm Im} \left( \lambda_{a,b}\right)\right). 
	\end{align*}
	We aim to find the maximal absolute value of 
	\[f(t)=\frac{1}{(1+t^2)^{\frac14}}\left( \left| \cos\left(\frac{\arctan t}{2}\right)\right|  +\left| \sin  \left(\frac{\arctan t}{2}\right)\right| \right). \]
	Using the trigonometric identities
	\begin{equation*}
		\cos\left(\frac{\Theta}{2} \right) =\sqrt{\frac{1+\cos\Theta}{2}},\quad
		\sin\left(\frac{\Theta}{2} \right) =\sqrt{\frac{1-\cos\Theta}{2}},\quad
		\cos(\arctan t)=\frac{1}{\sqrt{1+t^2}}, 
	\end{equation*}
	as well as the fact that $ |\arctan t|<\frac{\pi}{2}, $ we obtain
	\[f(t)=\frac{1}{\sqrt2}\left( \sqrt{\frac{1}{\sqrt{1+t^2}}+\frac{1}{1+t^2}}+\sqrt{\frac{1}{\sqrt{1+t^2}}-\frac{1}{1+t^2}} \right) ,\]
	and an easy calculus exercise shows that the maximum value of $ f $ occurs for $ t=\pm\frac1{\sqrt3} $ and equals \[f\left( \pm \frac1{\sqrt3}\right) =\frac{3^{\frac34}}{2} =1.13975\ldots<1.14.\]
	On noting that $\lambda_{0,1}=\frac{\Lambda_{0,1}}{2\sqrt2}=\frac{\Gamma\left( \frac32\right) \zeta\left( \frac32\right) }{2\sqrt{2}}$ and that by Lemma \ref{boundmaxReIm} there exists a small enough $ c>0 $ such that 
	\[{\rm{Re}}\left( \frac{\lambda_{a,b}}{\sqrt{\tau'}} \right)\le\frac{\lambda_{0,1}-c}{\sqrt y},\]
	we conclude the proof.  
\end{proof}
\begin{proof}[Proof of Lemma \ref{boundGbigoh}] If we are on a minor arc, then it suffices to apply Lemma \ref{minors} (because, as $ y\to 0, $ a negative power of $ y $ will dominate any positive power of $ y $), so let us assume that we are on a major arc.\\
	We first consider the behavior near 0, which corresponds to $ a=0, $ $ b=1,$ $ \tau=\tau'=y-2\pi ix. $ 
	Writing $ y^{\beta}=y^{\frac54-\varepsilon} $ with $ \varepsilon>0 $ (here we use the second inequality from \eqref{ineqbeta}), we have, on setting $ b=1 $ in \eqref{ineqxprim}, 
	\begin{equation}
		\label{cusp0}
		y^{\frac54-\varepsilon}\le |x|=|x'|\le y^{\frac23}.
	\end{equation}
	By \eqref{formforG} we get
	\[G(q)=Ce^{\frac{\Lambda_{0,1}}{2\sqrt2\sqrt{\tau}}}\frac{P_{0,1}(\tau)P_{0,1}(8\tau)}{P_{0,1}(4\tau )^2}  \]
	for some $ C>0 $ and thus, by Lemma \ref{majorarcs},
	\[\log|G(q)|=\frac{\Lambda_{0,1}}{2\sqrt2\sqrt{|\tau|}}+O(1).\]
	On using \eqref{cusp0} to prove the first inequality below and expanding into Taylor series to prove the second one, we obtain, by letting $ y\to0, $
	\[\frac{1}{\sqrt{|\tau|}}=\frac{1}{\sqrt{y}}\frac{1}{\left( 1+\frac{4\pi^2x^2}{y^2}\right)^{\frac14} }\le \frac{1}{\sqrt{y}}\frac{1}{\left( 1+4\pi^2y^{\frac12-2\varepsilon}\right) ^{\frac14}}\le \frac{1}{\sqrt{y}}\left( 1-c_6y^{\frac12-2\varepsilon}\right) \] for some $ c_6>0, $
	and this concludes the proof in this case. 
	
	To finish the claim we assume $ 2\le b\le y^{-\frac13}. $ If $ \tau\in\mathfrak M_{a,b}, $ then by \eqref{formforG} and Lemma \ref{majorarcs} we obtain 
	\[\log|G(q)|={\rm Re} \left( \frac{\lambda_{a,b}}{\sqrt{\tau'}} \right)+O\left( y^{-\frac13}\right)    \]
	as $ y\to0. $ Since by Lemma \ref{biglemma} there exists $ c_7>0 $ such that 
	\begin{equation}\label{ND}
		{\rm Re} \left( \frac{\lambda_{a,b}}{\sqrt{\tau'}} \right)\le \frac{\lambda_{0,1}}{\sqrt y}-\frac{c_7}{\sqrt{y}},
	\end{equation}
	we infer from \eqref{ND} that, as $ y\to 0, $  we have
	\[\log|G(q)|\le \frac{\lambda_{0,1}}{\sqrt y}-\frac{c_8}{\sqrt{y}}  \]
	for some $ c_8>0 $
	and the proof is complete. \end{proof}

\section{Proof of the Main Theorem}\label{Proof of Theorem 2} 
We have now all necessary ingredients to prove Theorem \ref{Conj1}, whose statement we repeat for convenience.
\begin{Thm}~
	\begin{enumerate}[{\rm (i)}] 	
		\item[\textup{(i)}] As $ n\to\infty, $ we have \[p_2(0,2,n)\sim p_2(1,2,n).\] 
		\item[\textup{(ii)}] Furthermore, for $ n $ sufficiently large, we have
		\[\begin{cases}
		p_2(0,2,n)>p_2(1,2,n) & \text{if~$ n $ is even,}\\
		p_2(0,2,n)<p_2(1,2,n) & \text{if~$ n $ is odd.}
		\end{cases}\]\end{enumerate}
\end{Thm}
\begin{proof} We begin by proving part (ii).
	By Lemma \ref{boundGbigoh} and the fact that $ \Lambda_{0,1}=\Gamma\left( \frac32\right)\zeta\left( \frac32\right),  $ we have 
	\begin{align}\label{correctbounderror}
		R(n)&=e^{ny} \int_{y^{\beta}\le|x|\le\frac12}  G\left( e^{-y+2\pi ix}\right) e^{-2\pi inx}dx\nonumber\\&\ll e^{ny} \int_{y^{\beta}\le|x|\le\frac12} e^{\frac{1}{2\sqrt2}\Gamma\left( \frac{3}{2}\right) \zeta\left( \frac32\right)\frac{1}{\sqrt{y}}-cy^{-\varepsilon}} dx\nonumber\\&\le e^{ny+\frac{1}{2\sqrt2}\Gamma\left( \frac32\right) \zeta\left( \frac32\right) \frac{1}{\sqrt y}-cy^{-\varepsilon}}=e^{3n^{\frac13}\left( \frac{1}{4\sqrt2}\Gamma\left( \frac32\right) \zeta\left( \frac32\right) \right)^{\frac23} -Cn^{\varepsilon_1}},
	\end{align}
	with $ \varepsilon_1=\frac{2\varepsilon}{3}>0 $ and some $ C>0. $
	\par We next turn to the asymptotic main term integral. Let $ n\ge  n_1 $ be large enough so that $y^{\beta-1}\le\frac{1}{2\pi}. $ This choice allows us to apply Lemma \ref{smallxmaintermG}, as it ensures $ |x|\le\frac12 $ and $ |\Arg (\tau)|\le \frac{\pi}{4}. $  Recalling that $ \Gamma \left( \frac32\right)  =\frac{\sqrt{\pi}}2,$ we obtain
	\begin{equation}\label{correctmainterm}
		e^{ny} \int_{-y^{\beta}}^{y^{\beta}}  G(e^{-y+2\pi ix})e^{-2\pi inx}dx= \frac{e^{ny}}{\sqrt2} \int_{-y^{\beta}}^{y^{\beta}}  e^{\frac1{4\sqrt2}\Gamma\left( \frac12\right)\zeta\left(\frac32 \right)\frac{1}{\sqrt{\tau}} +O(y^{\varepsilon})-2\pi inx }dx. \end{equation}
	Splitting 
	\[\frac{1}{\sqrt{\tau}}=\frac{1}{\sqrt{y}}+\left( \frac{1}{\sqrt{\tau}}-\frac{1}{\sqrt{y}}\right), \]
	we can rewrite \eqref{correctmainterm} as 
	\begin{align*}
		&\phantom{=~}e^{ny} \int_{-y^{\beta}}^{y^{\beta}}  G(e^{-y+2\pi ix})e^{-2\pi inx}dx\\&= \frac{e^{ny}}{\sqrt{2}}\int_{-y^{\beta}}^{y^{\beta}} e^{\frac{1}{4\sqrt2}\Gamma\left( \frac12\right) \zeta\left( \frac32\right)\frac{1}{\sqrt y}} e^{\frac{1}{4\sqrt2}\Gamma\left( \frac12\right) \zeta\left( \frac32\right)\left( \frac{1}{\sqrt{\tau}}-\frac1{\sqrt y}\right)  }e^{-2\pi inx+O(y^{c_0})}dx\\
		&= \frac{1}{\sqrt{2}}\int_{-y^{\beta}}^{y^{\beta}}\left( e^{ny+\frac{1}{4\sqrt2}\Gamma\left( \frac12\right) \zeta\left( \frac32\right)\frac{1}{\sqrt y}}\right) e^{\frac{1}{4\sqrt2}\Gamma\left( \frac12\right) \zeta\left( \frac32\right)\left( \frac{1}{\sqrt{\tau}}-\frac1{\sqrt y}\right)  }e^{-2\pi inx+O(y^{c_0})}dx\\
		&=\frac{e^{3n^{\frac13}\left( \frac{1}{4\sqrt2}\Gamma\left( \frac32\right) \zeta\left( \frac32\right) \right)^{\frac23} }}{\sqrt{2}}\int_{-y^{\beta}}^{y^{\beta}}e^{\frac{1}{2\sqrt2}\Gamma\left( \frac32\right) \zeta\left( \frac32\right)\frac{1}{\sqrt y}\left( \frac{1}{\sqrt{1-\frac{2\pi ix}{y}}}-1\right)  }e^{-2\pi inx+O(y^{c_0})}dx.
	\end{align*} 
	Putting $ u=-\frac{2\pi x}{y}, $ we get 
	\begin{align}\label{integralfinal}
		&\phantom{=~}e^{ny} \int_{-y^{\beta}}^{y^{\beta}}  G(e^{-y+2\pi ix})e^{-2\pi inx}dx\nonumber\\&=\frac{ye^{3n^{\frac13}\left( \frac{1}{4\sqrt2}\Gamma\left( \frac32\right) \zeta\left( \frac32\right) \right)^{\frac23} }}{2\sqrt2\pi}\int_{- 2\pi y^{\beta-1}}^{2\pi y^{\beta-1}}e^{\frac{1}{2\sqrt2}\Gamma\left( \frac{3}{2}\right) \zeta\left( \frac32\right)\frac{1}{\sqrt y}\left( \frac{1}{\sqrt{1+iu}}-1\right) +inuy +O(y^{c_0} ) }du.
	\end{align}
	Set $ B=\frac{1}{2\sqrt2}\Gamma\left( \frac32\right) \zeta\left( \frac32\right). $ 
	We have the Taylor series expansion
	\[
	\frac{1}{\sqrt{1+iu}}=1-\frac{iu}{2}-\frac{3u^2}{8}+\frac{5iu^3}{16}+\cdots=1-\frac{iu}{2}-\frac{3u^2}{8}+O(|u|^3), \]
	thus \[B\frac{1}{\sqrt y}\left( \frac{1}{\sqrt{1+iu}}-1\right)+inuy=-\frac{Biu}{2\sqrt y}+inuy-\frac{3Bu^2}{8\sqrt y}+O\left(\frac{|u|^3}{\sqrt y} \right). \]
	An easy computation shows that for $ y $ chosen as in \eqref{choicey} we have $B=2ny^{\frac32},$
	hence \[-\frac{Biu}{2\sqrt y}+inuy=0,  \]
	and, using \eqref{choicey} and the fact that $ |u|\le2\pi y^{\beta-1}, $ we obtain
	\begin{equation*}
		\label{correctBtaylor}
		B\frac{1}{\sqrt y}\left( \frac{1}{\sqrt{1+iu}}-1\right)+inuy=-\frac{3Bu^2}{8\sqrt y}+O\left(\frac{|u|^3}{\sqrt y} \right)=-\frac{3Bu^2}{8\sqrt y}+O\left(n^{\frac13\left(1+\frac{3(1-\beta)}{\alpha} \right) }\right).
	\end{equation*}
	Thus, if $ C_1= 2\pi\left( \frac{B}{2n}\right)^{\frac23(\beta-1)},$ we may change the integral from the right-hand side of \eqref{integralfinal} into
	\begin{align*}
		&\phantom{=~}\int_{|u|\le 2\pi y^{\beta-1}}e^{B\frac{1}{\sqrt y}\left( \frac{1}{\sqrt{1+iu}}-1\right) +inuy +O(y^{c_0} ) }du
		\\&=\int_{|u|\le C_1 }e^{-\frac{3Bu^2}{8\sqrt y}}  e^{O\left( y^{c_0}+\frac{u^3}{\sqrt y}\right) }   du\\
		&=\int_{|u|\le C_1  }e^{-\frac{3Bu^2}{8\sqrt y}}  e^{O\left( n^{-\frac{2c_0}3}+ n^{\frac13+2(1-\beta)}\right) }   du\\
		&=\int_{|u|\le C_1 }e^{-\frac{3\sqrt[3]{2n}\sqrt[3]{B^2}u^2}{8}}\bigg( 1+\bigg( e^{O\big( n^{-\frac{2c_0}3}+n^{\frac13+2(1-\beta)}\big) }-1\bigg) \bigg)    du.
	\end{align*}
	From \eqref{defbeta}, or equivalently, from the first inequality in \eqref{ineqbeta}, we now infer that $$ \frac13+2(1-\beta)=-\frac16+\frac{\delta}4<0 $$ and thus  
	\[e^{O\big( n^{-\frac{2c_0}3}+n^{\frac13+2(1-\beta)}\big) }-1=e^{O\big( n^{-\frac{2c_0}3}+n^{-\frac16+\frac{\delta}4}\big) }-1=O\left(n^{-\kappa} \right),  \]
	where $ \kappa=\min\left\lbrace \frac{2c_0}3,\frac16-\frac{\delta}4\right\rbrace . $ We further get
	\begin{equation*}
		\int_{|u|\le 2\pi y^{\beta-1}}e^{B\frac{1}{\sqrt y}\left( \frac{1}{\sqrt{1+iu}}-1\right) +inuy +O(y^{c_0} ) }du
		=\int_{|u|\le C_1  }e^{-\frac{3\sqrt[3]{2n}\sqrt[3]{B^2}u^2}{8}}\left(  1+ O\left( n^{-\kappa}\right) \right)  du
	\end{equation*}
	and, on using \eqref{choicey} again and setting $ v=\frac{\sqrt3\sqrt[6]{2n}\sqrt[3]Bu}{2\sqrt2} $ and $ C_2=2^{\frac13-\frac23\beta}\sqrt3\pi B^{\frac{2}{3}\beta-\frac13}>0, $ we obtain  
	\begin{align}\label{Gauss}
		&\phantom{=~}\int_{|u|\le 2\pi y^{\beta-1}}e^{B\frac{1}{\sqrt y}\left( \frac{1}{\sqrt{1+iu}}-1\right) +inuy +O(y^{c_0} ) }du\nonumber\\&=\int_{|u|\le C_1  }e^{-\frac{3\sqrt[3]{2n}\sqrt[3]{B^2}u^2}{8}}\left(  1+ O\left( n^{-\kappa}\right) \right)  du\nonumber\\
		&=\frac{2\sqrt2}{\sqrt3 \sqrt[6]{2n}\sqrt[3]B}\int_{|v|\le C_2\cdot n^{\frac{\delta}{12}}} e^{-v^2}\left(  1+ O\left( n^{-\kappa}\right) \right)dv.
	\end{align}
	By letting $ n\to\infty, $ we turn the integral from \eqref{Gauss} into a Gauss~integral. This introduces an exponentially small error and yields 
	\[\int_{|u|\le 2\pi y^{\beta-1}}e^{B\frac{1}{\sqrt y}\left( \frac{1}{\sqrt{1+iu}}-1\right) +inuy +O(y^{c_0} ) }du=\frac{2\sqrt2}{\sqrt3 \sqrt[6]{2n}\sqrt[3]B}\cdot \sqrt{\pi}\left(  1+ O\left( n^{-\kappa_1}\right) \right),\]
	where $ \kappa_1=\min\left\lbrace\frac{2c_0}{3}-\frac{\delta}{12},\frac16-\frac{\delta}3 \right\rbrace . $ Putting together \eqref{epxressionfora2}, \eqref{correctbounderror} and \eqref{integralfinal}, we obtain that, as $ n\to\infty, $
	the main asymptotic contribution for our coefficients $ a_2(n) $ is given by
	\begin{align}\label{p2diff}
		a_2(n)&\sim \frac{y}{2\sqrt2\pi}\cdot e^{3n^{\frac13}\left( \frac{1}{4\sqrt2}\Gamma\left( \frac32\right)\zeta\left( \frac32\right)   \right)^{\frac23} }\cdot\frac{2\sqrt2}{\sqrt3\sqrt[6]{2n} \sqrt[3]B}\int_{-\infty}^{\infty} e^{-v^2}dv\nonumber\\
		&=\frac{y}{2\sqrt2\pi}\cdot e^{3n^{\frac13}\left( \frac{1}{4\sqrt2}\Gamma\left( \frac32\right)\zeta\left( \frac32\right)   \right)^{\frac23} }\cdot\frac{2\sqrt2\cdot\sqrt{\pi}}{\sqrt3\sqrt[6]{2n} \sqrt[3]B}\nonumber\\
		&=\frac{\sqrt[3]B}{\sqrt{3\pi}\cdot (2n)^{\frac56} }e^{3n^{\frac13}\left( \frac{1}{4\sqrt2}\Gamma\left( \frac32\right)\zeta\left( \frac32\right)   \right)^{\frac23} }.
	\end{align}
	This shows that $ a_2(n)>0 $ as $ n\to\infty $, hence part (ii) of Theorem \ref{Conj1} is proven.
	
	\par We now turn to  part (i). Clearly, $ p_2(n)=p_2(0,2,n)+p_2(1,2,n). $   By applying either Meinardus' Theorem (Theorem \ref{Meinardus}) or Wright's Theorem (\cite[Theorem 2]{WrightIII}) we have, on keeping the notation from \cite[pp. 144--145]{WrightIII},
	\begin{equation*}
		p_2(n)\sim B_0n^{-\frac76}e^{\Lambda n^{\frac13}},
	\end{equation*}
	where 
	\[\quad B_0=\frac{\Lambda}{2\cdot(3\pi)^{\frac32}}\quad\text{and}\quad \Lambda=3\left(\frac{\Gamma\left(\frac32 \right)\zeta\left( \frac32\right)}{2}   \right)^{\frac23}=6\left(\frac{1}{4\sqrt2} \Gamma\left(\frac32 \right)\zeta\left( \frac32\right)  \right)^{\frac23}. \]
	We thus obtain 
	\begin{equation}\label{p2wright}
		p_2(n)\sim B_0n^{-\frac76}e^{6n^{\frac13}\left( \frac{1}{4\sqrt2}\Gamma\left( \frac32\right)\zeta\left( \frac32\right)   \right)^{\frac23} }.
	\end{equation}
	On adding \eqref{p2diff} and \eqref{p2wright} and recalling that 
	\begin{equation*}\label{a2fromp2}
		a_2(n)=\begin{cases}
			p_2(0,2,n)-p_2(1,2,n) & \text{if~$ n $~is even,}\\
			p_2(1,2,n)-p_2(0,2,n) & \text{if~$n$~is odd,}
		\end{cases}
	\end{equation*}
	we have   
	\[p_2(0,2,n)\sim p_2(1,2,n)\sim \frac{B_0}{2}n^{-\frac76}e^{6n^{\frac13}\left( \frac{1}{4\sqrt2}\Gamma\left( \frac32\right)\zeta\left( \frac32\right)   \right)^{\frac23} } \]
	as $ n\to\infty, $ and the proof is complete.
\end{proof}

\begin{Rem}
	As promised at the beginning and already revealed by our proof, by plugging in the values of $ B_0 $ and $\Lambda$ we obtain, as $ n\to\infty, $ the asymptotics  \[p_2(0,2,n)\sim p_2(1,2,n)\sim \frac{1}{2\pi\sqrt{3\pi}}\left(\frac{1}{4\sqrt2} \Gamma\left(\frac32 \right)\zeta\left( \frac32\right)  \right)^{\frac23} n^{-\frac76}e^{6n^{\frac13}\left(  \frac{1}{4\sqrt2}\Gamma\left( \frac32\right)\zeta\left( \frac32\right)\right)^{\frac23}}.\]
\end{Rem}
\begin{Rem}
	Note that, although we could not apply Meinardus' Theorem to our product in \eqref{GfromH}, the asymptotic value we obtained for $ a_2(n) $ in \eqref{p2diff} agrees, surprisingly or not, precisely with that given for $ r(n) $ in Theorem \ref{Meinardus}. This indicates that, even if it may not directly apply to certain generating products, Meinardus' Theorem is a powerful enough tool to provide correct heuristics. 
\end{Rem}
\begin{Rem} We notice that, in its original formulation, part (ii) of Conjecture \ref{ConjBM} is not entirely true since there are cases when $ p_2(0,2,n)=p_2(1,2,n), $ as it happens, e.g., for $ n\in\{ 4,5,6,7,13,14,15,16,22,23,24,\allowbreak 31,39,47,48,56,64\}. $ No other values of $ n $ past 64 revealed such pattern and, based on the behavior we observed, we strongly believe that the inequalities hold true for $n\ge 65.$ In particular, we checked this is the case up to $ n=50,000. $ 
\end{Rem}

\section*{Acknowledgments} The author would like to thank Kathrin Bringmann and Karl Mahlburg for suggesting this project and for many helpful discussions, much of the present content originating in their initial efforts on this work. The author is also grateful to Stephan Ehlen for useful discussions, to Chris Jennings-Shaffer for his extensive comments on previous versions of this paper and support with some MAPLE computations, and to the anonymous referee for various suggestions on improving the presentation. The work was supported by the European Research Council under the European Union's Seventh Framework Programme (FP/2007--2013) / ERC Grant agreement n. 335220 --- AQSER.


\begin{thebibliography}{99}
	
	\bibitem{Apostol} T. M. Apostol, \textit{Introduction to Analytic Number Theory}, Undergraduate Texts in Mathematics. Springer-Verlag, New York-Heidelberg, 1976.
	\bibitem{Apostol2} T. M. Apostol, \textit{Modular Functions and Dirichlet Series in Number Theory}, Second edition, Graduate Texts in Mathematics 41. Springer-Verlag, New York, 1990.
	\bibitem{And} G. E. Andrews, \textit{The theory of partitions}, Reprint of the 1976 original, Cambridge Mathematical Library. Cambridge University Press, Cambridge, 1998.
	\bibitem{AAR} G. E. Andrews, R. Askey and R. Roy, \textit{Special functions}, Encyclopedia of Mathematics and its Applications 71. Cambridge University Press, Cambridge, 1999.
	\bibitem{Berndt} B. C. Berndt, R. J. Evans and K. S. Williams, \textit{Gauss and Jacobi sums}, Canadian Mathematical Society Series of Monographs and Advanced Texts. A Wiley-Interscience Publication. John Wiley \& Sons, Inc., New York, 1998.
	\bibitem{BM} K. Bringmann and K. Mahlburg, Transformation laws and asymptotics for nonmodular products (unpublished preprint).
	\bibitem{Glaisher} J. W. L. Glaisher, On formulae of verification in the partition of numbers, \textit{Proc. Royal Soc. London} \textbf{24} (1876), 250--259.
	\bibitem{Mein} G. Meinardus, Asymptotische aussagen \"uber Partitionen, \textit{Math. Z.} \textbf{59} (1954), 388--398.
	\bibitem{Ram}  S. Ramanujan, Congruence properties of partitions, \textit{Math. Z.} \textbf{9} (1921), 147--153.
	\bibitem{Titch} E. C. Titchmarsh, \textit{The theory of the Riemann zeta-function}, Second edition. Edited and with a preface by D. R. Heath-Brown, The Clarendon Press, Oxford University Press, New York, 1986.
	\bibitem{AMM} L. T\'oth, Elementary Problems: E 3432, \textit{Amer. Math. Monthly} \textbf{98} (1991), no. 3, 263--264. 	
	\bibitem{WrightIII} E. M. Wright, Asymptotic partition formulae. III. Partitions into $k$-th powers, \textit{Acta Math.} \textbf{63} (1) (1934), 143--191.
	
\end{thebibliography}
\end{document}